\documentclass[12pt]{article}
\usepackage[utf8]{inputenc}

\usepackage[margin=1in]{geometry}
\usepackage{amsfonts,amssymb,amsthm,amsmath,graphicx}
\usepackage{mathrsfs}
\usepackage{float}
\usepackage{listings}
\usepackage{hyperref}
\usepackage{thmtools, thm-restate, cleveref}
\usepackage{verbatim}
\usepackage{mathtools}
\usepackage{enumitem}
\usepackage[nolabel, final]{showlabels}
\usepackage{xcolor}
\usepackage{bm} 
\usepackage{caption}
\usepackage{multirow}
\hypersetup{
    colorlinks,
    linkcolor={red!50!black},
    citecolor={blue!50!black},
    urlcolor={blue!80!black}
}

\newcommand{\norm}[1]{\left\lVert#1\right\rVert}

\newcommand{\E}[1]{\mathbb{E}\left[ #1 \right]}

\newcommand{\mtx}[1]{\bm{#1}}
\newcommand{\A}{\mtx{A}}
\def \B {\mtx{B}}
\def \Z {\mtx{Z}}
\def \X {\mtx{X}}
\def \Y {\mtx{Y}}
\def \I {\mtx{I}}

\newtheorem{theorem}{Theorem}
\newtheorem{corollary}{Corollary}[theorem]

\newtheorem{prop}{Proposition}
\newtheorem{remark}{Remark}

\definecolor{boxbackground}{rgb}{1,0.976, 0.882}

\title{Concentration inequalities for random matrix products}
  
\author{Amelia Henriksen and Rachel Ward\thanks{Oden Institute for Computational Engineering and Sciences, University of Texas, Austin, TX (email: amelia@ices.utexas.edu, rward@math.utexas.edu).  This material is based upon work supported in part by AFOSR MURI Award N00014-17-S-F006.}}

\begin{document}
\maketitle

\begin{abstract}
Suppose $\{\X_k\}_{k \in \mathbb{Z}}$ is a sequence of bounded independent random matrices with common dimension  $d\times d$ and common expectation $\E{ \X_k }= \X$.  Under these general assumptions, the normalized random matrix product
$$\Z_n = (\I_d + \frac{1}{n}\X_n)(\I_d + \frac{1}{n}\X_{n-1}) \cdots (\I_d + \frac{1}{n}\X_1)$$
converges to $\Z_n \rightarrow e^{\X}$ as $n \rightarrow \infty$.  Normalized random matrix products of this form arise naturally in stochastic iterative algorithms, such as Oja's algorithm for streaming Principal Component Analysis.
Here, we derive \emph{nonasymptotic concentration inequalities} for such random matrix products.  In particular, we show that the spectral norm error satisfies $\| \Z_n - e^{\X} \| = O((\log(n))^2\log(d/\delta)/\sqrt{n})$ with probability exceeding $1-\delta$. This rate is sharp in $n$, $d$, and $\delta$, up to possibly the $\log(n)$ and $\log(d)$ factors. 
The proof relies on two key points of theory: the Matrix Bernstein inequality concerning the concentration of sums of random matrices, and Baranyai's theorem from combinatorial mathematics.
Concentration bounds for general classes of random matrix products are hard to come by in the literature, and we hope that our result will inspire further work in this direction. 
\end{abstract}

\section{Introduction}
A classical limit theorem from complex analysis reads: \emph{Let $(u_n)_{n \in \mathbb{N}}$ be a uniformly bounded complex sequence whose mean $\frac{1}{n} \sum_{k=0}^{n-1} u_n$ converges towards $\mu$.  Then  }
\begin{equation}
\label{eq:classical limit law}
\lim_{n \rightarrow \infty} \displaystyle \prod \limits_{k=0}^{n-1} \left(1 + \frac{u_k}{n} \right) = e^{\mu}.
\end{equation}
This result is easily verified by taking the natural logarithm of each side, and observing that 
$\log \left( \prod \limits_{k=0}^{n-1} \left(1 + \frac{u_k}{n} \right) \right) \approx \frac{1}{n} \sum_{k=0}^{n-1} u_k \rightarrow \mu$.  A non-commutative extension of this result was recently proven by Emme and Hubert in \cite{limitlaw2018}:
\begin{prop}
\label{limitlaw}
Let $(\A_n)_{n \in \mathbb{N}}$ be a sequence of $d \times d$ complex matrices satisfying 
$$
\lim_{n \rightarrow \infty}  \frac{1}{n} \sum_{k=1}^{n} \A_k = \A
$$ 
and such that $(\frac{1}{n} \sum_{k=1}^{n} \| \A_k \| )_{n \in \mathbb{N}}$ is bounded for a norm $\| \cdot \|$ by $\alpha$.  Consider the matrix product
$$
\Z_n = \left( \I_d + \frac{1}{n} \A_1 \right) \dots \left( \I_d + \frac{1}{n} \A_n \right).
$$
Then 
$$
\lim_{n \rightarrow \infty} \Z_n = e^{\A}.
$$
\end{prop}
The proof of Theorem \ref{limitlaw} is not a straightforward extension of the scalar result.  The matrix product is non-commutative in general, $\A\B \neq \B \A$, and so of course $\log(\A\B) \neq \log(\A) + \log(\B)$ fails to hold in turn.  

%The proof of \ref{limitlaw} from \cite{limitlaw2018} begins by writing $\Z_n$ as follows,
%$$
%\Z_n = \sum_{k=0}^{n-1} \left(\frac{1}{n} \right)^k \left( \sum_{0 \leq i_1 \leq \dots \leq i_k \leq n-1} A_{i_1} \dots A_{i_k} \right) 
%$$
%and analyzing the convergence of each summand $\left(\frac{1}{n} \right)^k \left( \sum_{0 \leq i_1 \leq \dots \leq i_k \leq n-1} A_{i_1} \dots A_{i_k} \right) $ separately.  

\noindent An important special case within the framework of Proposition \ref{limitlaw} is when the $\A_k$ are uniformly bounded independent random matrices with common expectation $\E{\A_k} = \A$.  Then $\Z_n$ is also a random matrix, and has expectation $\E{\Z_n} = (\I_d + \frac{1}{n} \A)^n$.  Within this framework, it is natural to ask about about \emph{rates of convergence} of $\Z_n$ to $e^{\A}$.   As far as we are aware, precise rates of convergence for matrix products of the form $\Z_n$ have not appeared in the literature before, despite such random matrix products naturally arising in stochastic iterative algorithms such as stochastic gradient descent; in particular, in Oja's algorithm for estimating the top eigenvector of the covariance matrix of a distribution of matrices observed sequentially    \cite{streaming1, streaming2, streaming3, streaming4, streaming5, 2016Jain, allen2017first}.    Here, as the main content of this paper, we derive a rate of convergence for matrix products of this form.

\begin{theorem}[Main Theorem]
\label{thm:main}
Consider a sequence $\{ \X_k \}_{k \in \mathbb{Z}}$ of independent (real or complex-valued) random matrices with common dimension $d \times d$.  Assume that
$$
\E{ \X_k } = \X \quad \quad \text{and} \quad \quad \| \X_k \| \leq L \quad \text{for each index } k.
$$
Introduce the sequence of random matrices $\{ \Z_n \}_{n \in \mathbb{N}}$ given by 
\begin{equation}
\label{eq:Zn}
\Z_n = (\I_d + \frac{1}{n}\X_n)(\I_d + \frac{1}{n}\X_{n-1}) \cdots (\I_d + \frac{1}{n}\X_1).
\end{equation}
Suppose that $L > 0$, $n, d \in \mathbb{Z},$ and $\delta \in (0, 1/2]$ are such that
\begin{equation}
\label{eq:restriction1}
\max\{3, \lceil{ L e^2 \rceil} \} \leq  \log(n) +1 \leq \left( \frac{16 n}{\log(d/\delta) + \log(ne)} \right)^{1/3}
\end{equation}  Then with probability exceeding $1 - 2\delta$, the following holds:
\begin{align*}
\norm{\Z_n - e^{\X}} &\leq \frac{2L  e^L \log(n)}{\sqrt{n}} \left( 2 \sqrt{ \log( 2d/\delta) + (\log(n))^2} +  \frac{\log(n)}{\sqrt{n}} \right)  + \frac{ L^2 e^L }{2n}
\end{align*}
where $\| \cdot \|$ denotes the matrix spectral norm.
\end{theorem}
\noindent Theorem \ref{thm:main} immediately implies a bound on the expected value of $\norm{\Z_n - e^{\X}}.$  Note that 
$\| \Z_n \| \leq e^L$ and $ \| e^{\X} \| \leq e^L$, so for any $\delta > 0$ satisfying \eqref{eq:restriction1}, 
\begin{equation}
\mathbb{E} \norm{\Z_n - e^{\X}} \leq (1-2\delta) \left( \frac{2L  e^L \log(n)}{\sqrt{n}} \left( 2 \sqrt{ \log( 2d/\delta) + (\log(n))^2} +  \frac{\log(n)}{\sqrt{n}} \right)  \right) + \frac{ L^2 e^L }{2n} 
   + 4\delta e^L \nonumber
\end{equation}
In particular, setting $\delta = \frac{L^2}{8n}$ gives 
$$
\mathbb{E} \norm{\Z_n - e^{\X}} \leq \left( \frac{2L  e^L \log(n)}{\sqrt{n}} \left( 2 \sqrt{ \log( 2d/\delta) + (\log(n))^2} +  \frac{\log(n)}{\sqrt{n}} \right)  \right) + \frac{ L^2 e^L }{n}.
$$

\noindent Note that the $O(\frac{1}{\sqrt{n}})$ convergence rate is unavoidable under the stated assumptions.  Indeed, consider the scalar case $d=1$, where $\{ x_j\}_{j=1}^n $ is a sequence of independent real-valued mean-zero scalars, bounded uniformly by $| x_j | \leq L$.  In this case, as $\frac{1}{n} \sum_{j=1}^n x_j $ becomes sufficiently small, $ \frac{1}{n} \sum_{j=1}^n x_j$ and $\log( \prod \limits_{j=1}^n (1 + \frac{x_j}{n}) )$ are nearly equivalent.  Thus, applying the standard scalar Bernstein inequality to $ \frac{1}{n} \sum_{j=1}^n x_j$ results in a bound of the form $|  \prod \limits_{j=1}^n (1 + \frac{x_j}{n})    - 1  | \leq C L \frac{\sqrt{\log(1/\delta)}}{\sqrt{n}}$. It remains open whether the $\log(n)$ and $\log(d)$ factors in the rate given by Theorem \ref{thm:main} can be removed, and also whether the dependence on $L$ can be improved.

\begin{remark} \emph{ Limit laws for products of random matrices have been extensively analyzed in the context of ergodic theory or martingales on Markov chains -- see for instance the book \cite{2016Benoist} or the extensive survey articles \cite{2002Furman, 2001Ledrappier}.  However, results in the form of \emph{quantitative rates of convergence} of general random matrix products are quite scarce, apart from specialized cases such as for products of i.i.d. \emph{Gaussian} random matrices.  Surprisingly, for the random matrix product $\Z_n$ we consider \eqref{eq:Zn}, a rigorous proof of the limiting behavior $\Z_n \rightarrow \exp(X)$ appears to have only been proven recently \cite{limitlaw2018}, even though a seemingly incomplete proof of this limiting behavior was provided as Theorem $7$ of the 1984 paper \cite{clt1984}.}
\end{remark}

\noindent {\bf Notation.} Throughout,  $\| \X\|$ refers to the spectral norm of the matrix $\X$.  For an integer $n \geq 1$, we use the notation $[n]$ to refer to the set $\{1, 2, \dots, n\}$.  We write ${\bf Prob}[E]$ to refer to the probability of the event $E$.  

\section{Preliminaries}

A crucial ingredient of the proof of Theorem \ref{thm:main} is the matrix Bernstein inequality, a matrix-level extension of the classical scalar Bernstein inequality describing the upper tail of a sum of independent bounded or sub-exponential random variables.  The first matrix Bernstein type bound was derived by Ahlswede and Winter \cite{aw2003}, and subsequently improved by Tropp \cite{tropp2010} by applying Lieb's theorem in place of the Golden-Thompson inequality.     We use the variant of the matrix Bernstein inequality of Tropp stated below.

\begin{prop}[Matrix Bernstein Inequality (Theorem 6.1.1 in \cite{2015Tropp})]
\label{prop:bernstein}
Consider a finite sequence $\{ \mtx{S}_k \}$ of independent random matrices with common dimension $d_1 \times d_2$.  Assume that
$$
\E{ \mtx{S}_k } = 0 \quad \quad \text{and} \quad \quad \| \mtx{S}_k \| \leq L \quad \text{for each index } k.
$$
Introduce the random matrix 
$$
\mtx{Z} = \sum_k \mtx{S}_k.
$$
Let $v(\mtx{Z})$ be the matrix variance statistic of the sum:
\begin{align}
v(\mtx{Z}) &= \max\{ \| \E{ \mtx{Z} \mtx{Z}^{*} } \|, \| \E{ (\mtx{Z}^{*} \mtx{Z})} \| \}  \\
&= \max\{ \| \sum_k \E{ \mtx{S}_k \mtx{S}_k^{*} } \|, \| \sum_k \E{ \mtx{S}_k^{*} \mtx{S}_k} \|  \}.
\end{align}
Then, for all $t \geq 0$,
$$
{\bf Prob} \{ \| \mtx{Z} \| \geq t \} \leq (d_1 + d_2) \exp\left( \frac{-t^2/2}{v(\mtx{Z}) + Lt/3} \right).
$$
\end{prop}

Another key theorem we rely on is Baranyai's theorem \cite{1975Baranyai}, stated below.

\begin{prop}[Baranyai, 1973]
Let $a_1, \dots, a_t$ be natural numbers such that $\sum_{j=1}^t a_j = {N\choose k}$.  Then the set of $k$-subsets of $[N]$ can be partitioned into disjoint families $S_1, \dots, S_t$ with $| S_j | = a_j$ and each $i \in [N]$ is included in exactly $\lceil{ \frac{a_j \cdot k }{N} \rceil}$ or $\lfloor{ \frac{a_j \cdot k}{N} \rfloor}$ elements of $S_j$.  
\end{prop}

%\noindent \emph{ Sketch of the proof of Theorem \ref{thm:main}.}
\subsection{Sketch of the proof of Theorem \ref{thm:main}}

Suppose that $\X_k$, $\X$, and $\Z_n$ satisfy the assumptions of Theorem \ref{thm:main}.  
Write
\begin{align}
\Z_n &= \sum_{k=0}^n \left(\frac{1}{n}\right)^k \sum_{1 \leq j_1 < \cdots < j_k \leq n} \X_{j_k}\X_{j_{k-1}} \cdots \X_{j_1}  \\
&= \I_d  + \sum_{k=1}^n \Z_{n,k}  
\end{align}
where 
\begin{equation}
\label{eq:Znk}
\Z_{n,k} = \left(\frac{1}{n}\right)^k \sum_{1 \leq j_1 < \cdots < j_k \leq n} \X_{j_k}\X_{j_{k-1}} \cdots \X_{j_1}, \quad \quad 1 \leq k \leq n.
\end{equation}
Because the $\X_k$ are independent, the expected values of $\Z_{n,k}$ and $\Z_n$ are easily calculated:
\begin{equation}
\label{eq:exp}
\E{\Z_{n,k}} = \left(\frac{1}{n} \right)^k\binom{n}{k}\X^k, \quad \quad \E{\Z_n} = \I_d + \sum_{k=1}^n \E \Z_{n,k} = (\I_d + \frac{1}{n}\X)^n; \quad \quad  .
\end{equation}
We then write
\begin{align}
\label{eq:err}
\norm{\Z_n - e^{\X}} &\leq \norm{\Z_n - \mathbb{E}[\Z_n] } + \norm{ \mathbb{E}[\Z_n] - e^{\X}} \nonumber \\
&= \norm{\Z_n - \mathbb{E}[\Z_n] } + \norm{ (\I_d + \frac{1}{n}\X)^n- e^{\X}}  \\
&\leq \sum_{k=1}^n \norm{\Z_{n,k} - \mathbb{E}[\Z_{n,k}]} +  \norm{ (\I_d + \frac{1}{n}\X)^n- e^{\X}} 
\end{align}
The approximation error $ \norm{(\I_d + \frac{1}{n}\X)^n- e^{\X}}$ is bounded deterministically using standard analysis, and converges to zero at rate $O(1/n)$, as made precise by Lemma \ref{prop:4}.   The errors $\norm{\Z_{n,k} - \mathbb{E}[\Z_{n,k}]}$ decay sufficiently quickly in $k$ that the sum of all but the first $\log(n)$ many of them, $\sum_{k = \lceil{\log(n) \rceil}}^n \| \Z_{n,k} - \mathbb{E}[\Z_{n,k}]  \|$, is also bounded by $O(1/n)$ deterministically (Lemma \ref{prop:1} below).   The leading error term $\| \Z_{n,1} - \mathbb{E}[\Z_{n,1}]  \|$  is bounded with high probability using the Matrix Bernstein inequality.   The most interesting, and most difficult, part of the proof is in bounding the intermediate terms $\| \Z_{n,k} - \mathbb{E}[\Z_{n,k}] \|$, $k = 2, \dots, \lfloor{ \log(n) \rfloor}.$  To do this, we appeal to Baranyai's theorem, which implies that each such term can be approximately written as a sum of \emph{sums of independent matrix products}, so that we may apply the matrix Bernstein inequality with properly tuned parameters to each sub-sum to achieve the final bound.   

\section{Key Ingredients}

The first two lemmas use standard analysis tools;  we defer the proofs to appendices.  

\begin{restatable}{lemma}{eXbound}
\label{prop:4}
Let $\X$ be a square real or complex-matrix with spectral norm $\| \X \|$. The following holds:
$$
 \| (I + \frac{1}{n}\X)^n - e^{\X} \|  \leq \frac{\| \X \| ^2}{2n} e^{\| \X \|}
$$
\end{restatable}
\noindent The proof of Lemma \ref{prop:4} is found in Appendix \ref{app:3}.

\begin{restatable}{lemma}{tailbound}
\label{prop:1}
Suppose that $\Z_n$ is as in Theorem \ref{thm:main}, and let $\Z_{n,k}$ be as defined in \ref{eq:Znk}. Suppose that $\lceil \log(n) \rceil \geq \max\{3, \lceil{ L e^2 \rceil} \}.$ Then
$$
\sum_{k= \lceil \log(n) \rceil }^n \norm{\Z_{n,k} - \mathbb{E}\left(\Z_{n,k} \right)} \leq \frac{2Le^2}{n(e-1)}
$$
\end{restatable}
\noindent The proof of Lemma \ref{prop:1} is found in Appendix \ref{app:1}.

Proposition \ref{prop:2} contains the meat of the proof.  By carefully combining the Matrix Bernstein inequality and Baranyai's theorem,  
we produce high probability bounds for the error terms $ \| \Z_{n,k}  - \mathbb{E}\left(\Z_{n,k} \right) \|$.

\begin{restatable}{prop}{probbound}
\label{prop:2}
Assume $\X_1, \X_2, \ldots \X_n$ are $d \times d$ matrices satisfying the assumptions in Theorem \ref{thm:main}, and suppose that $n, k, d \in \mathbb{Z},$ and $\delta > 0$ are such that 
\begin{align}
\label{k:simplified}
k &\leq \left( \frac{16 n}{\log(d/\delta) + \log(ne)} \right)^{1/3}
\end{align}
where, for the $k=1$ case, we treat $0^0 = 1$. 
Then
 $${\bf Prob}[ \norm{\Z_{n,k} - \mathbb{E}(\Z_{n,k})} > \gamma_k] \leq \delta^k$$
 where 
 \begin{equation}
\label{gamma_k}
\gamma_k = 2\left( \frac{e L}{k-1} \right)^{k-1}  \left( \frac{2L}{\sqrt{n}} \sqrt{ \log\left(  \frac{2d (ne / (k-1))^{k-1}}{ \delta} \right) }+  \frac{L(k-1)}{n} \right)
\end{equation}
\end{restatable}

\begin{proof}

For simplicity of notation, we drop the subscript $n$ in all matrix notation throughout; that is, we let $\Z = \Z_n$, we let $\Z_{n,k} = \Z_k$, and so on. 
Note 
Let $p \in \{0,1, \ldots, k-1\}$ be the unique integer such that $k$ divides $n-p$, and write
\begin{align}
\Z_{k} - \E{\Z_{k}} &= \left(\frac{1}{n} \right)^k \sum_{1 \leq j_1 < \cdots < j_k \leq n} \left( \X_{j_k}\X_{j_{k-1}} \cdots \X_{j_1} - \X^k\right) \\
&=  \left(\frac{1}{n} \right)^k \sum_{1 \leq j_1 < \cdots < j_k \leq n-p} \left( \X_{j_k}\X_{j_{k-1}} \cdots \X_{j_1} - \X^k\right) + \mtx{D}_k.
\end{align}
The random matrix $\mtx{D}_k$ is a sum of $\binom{n}{k} - \binom{n-p}{k}$ random matrix products, each of which contains at least one of the $p$ matrices $\X_{n-p+1}, \cdots \X_n$.  Each term is bounded in norm deterministically by $2\left(\frac{L}{n}\right)^k$, so
\begin{align*}
\| \mtx{D}_k \| &\leq \left(\binom{n}{k} - \binom{n-p}{k} \right)2\left(\frac{L}{n}\right)^k\\
&\leq  2(k-1)\binom{n-1}{k-1}\left(\frac{L}{n}\right)^k \quad \quad \quad \text{(by Pascal's rule)} \\
&\leq 2(k-1)\left(\frac{(n-1)(e)}{k-1}\right)^{k-1}\left(\frac{L}{n}\right)^k\\
&\leq 2 \frac{L(k-1)}{n} \left(\frac{e L}{k-1} \right)^{k-1}\\
\end{align*}

We thus have so far that
\begin{equation}
\norm{\Z_k - \E{\Z_k}} \leq \norm{  \left(\frac{1}{n} \right)^k \sum_{1 \leq j_1 < \cdots < j_k \leq n-p} \left( \X_{j_k}\X_{j_{k-1}} \cdots \X_{j_1} - \X^k\right) } +  \frac{2L(k-1)}{n} \left(\frac{e L}{k-1} \right)^{k-1}. \nonumber
\end{equation}
Now, as a consequence of Baranyai's theorem, there exist $m_k = \frac{{n-p \choose k}}{(n-p)/k} = \binom{n-p}{k-1}$ partitions of $[n-p] = \{1,2, \dots, n-p\}$, denoted by ${\cal P}_r$, $r= 1,2,\dots, m_k$,  such that 
\begin{align}
\sum_{1 \leq j_1 < \cdots < j_k \leq n-p} \left( \X_{j_k}\X_{j_{k-1}} \cdots \X_{j_1} - \X^k\right) &= \sum_{r=1}^{m_k} \sum_{ \substack{ 1 \leq j_1 < \cdots < j_k \leq n-p, \\  \{j_1, \dots, j_k \} \in {\cal P}_r}}  \left(\X_{j_k}\cdots \X_{j_1} - \X^k\right) 
\end{align}
Write 
$$
\sum_{\ell=1}^{(n-p)/k}  \Y_{r,\ell} =    \sum_{ \substack{1 \leq j_1 < \cdots < j_k \leq n-p, \\  \{j_1, \dots, j_k \} \in {\cal P}_r}} \left(\frac{1}{n}\right)^k \left(\X_{j_k}\cdots \X_{j_1} - \X^k\right).
$$
Because the $\X_j$ are independent and because each ${\cal P}_r$ constitutes a partition of $[n-p]$, each subset of random matrices $\{ \Y_{r,\ell} \}_{\ell=1}^{(n-p)/k}$ forms a mutually independent set of random matrices.  We can use this to bound $\norm{\sum_{r=1}^{m_k} \sum_{\ell=1}^{(n-p)/k} \Y_{r, \ell}}$ with high probability, using the Matrix Bernstein Inequality (Proposition \ref{prop:bernstein}).  Indeed, we will apply the Matrix Bernstein Inequality separately to each sum $\sum_{\ell=1}^{(n-p)/k} \Y_{r,\ell}$ of independent random matrices.  To do this, we employ the bounds
\begin{enumerate}
    \item $\E{\Y_{r,\ell}} = \left(\frac{1}{n}\right)^k\E{\X_{j_k}\cdots \X_{j_1} - \X^k}  = \left(\frac{1}{n}\right)^k\left(\X^k - \X^k\right) =0$
    \item $\| \Y_{r,\ell} \|  = \norm{\left(\frac{1}{n}\right)^k\left(\X_{j_k}\cdots \X_{j_1} - \X^k\right)} \leq \left(\frac{1}{n}\right)^k\left(\norm{\X_{j_k}}\cdots \norm{\X_{j_1}} + \norm{X}^k\right) \leq  2\left(\frac{L}{n}\right)^k$
    \item 
    \begin{align*}
    \norm{\sum_{\ell=1}^{(n-p)/k}\E{\Y_{r,\ell} {\Y_{r,\ell}}^*}} &\leq \sum_{\ell=1}^{(n-p)/k}\norm{\E{\Y_{r,\ell} {\Y_{r,\ell}}^*}} \leq \sum_{\ell=1}^{(n-p)/k}\E{\norm{\Y_{r,\ell} {\Y_{r,\ell}}^*}} \leq \sum_{\ell=1}^{(n-p)/k}\E{\norm{\Y_{r,\ell}}^2}   \\
    &\leq \sum_{\ell=1}^{(n-p)/k}4\left(\frac{L}{n}\right)^{2k}  \leq 4 \left(\frac{n}{k}\right) \left(\frac{L}{n}\right)^{2k} \nonumber
    \end{align*}
    \item Similarly, $\norm{\sum_{\ell=1}^{(n-p)/k}\E{{\Y_{r,\ell}}^* \Y_{r,\ell}}} \leq 4\left(\frac{n}{k}\right)\left(\frac{L}{n}\right)^{2k} $
\end{enumerate}

We can now apply the Matrix Bernstein Inequality: for any $\tau > 0$,
\begin{align*}
{\bf Prob} \left( \norm{  \sum_{\ell=1}^{\frac{(n-p)}{k}} \Y_{r,\ell}} \geq \tau \right) &\leq 2d \exp{\left(\frac{-\tau^2/2}{4\frac{n}{k}\left(\frac{L}{n}\right)^{2k}  + 2\tau/3\left(\frac{L}{n}\right)^{k}} \right)}
\end{align*}
We take the union bound over all $m_k = \binom{n-p}{k-1} \leq  \binom{n}{k-1} \leq (\frac{ne}{k-1})^{k-1}$ sums to obtain
$${\bf Prob} \left(\exists r \in [m_k] : \norm{  \sum_{\ell=1}^{\frac{(n-p)}{k}} \Y_{r,\ell}} \geq \tau \right) \leq 2d  \left( \frac{ne}{k-1} \right)^{k-1} \exp{\left(\frac{-\tau^2/2}{4\frac{n}{k}\left(\frac{L}{n}\right)^{2k}  + \tau\left(\frac{L}{n}\right)^{k}} \right)}$$

Set $\tau = \beta_k  (\frac{ne}{k-1})^{-(k-1)}$ (where, in case $k=1$, we use $0^0 = 1$).
Then
\begin{align*}
{\bf Prob} \left(\sum_{r=1}^{m_k}\norm{  \sum_{\ell=1}^{\lfloor n/k \rfloor} \Y_{r,\ell}}  \geq \beta_k \right) &\leq {\bf Prob} \left(\exists r \in [m_k] : \norm{  \sum_{\ell=1}^{\lfloor n/k \rfloor} \Y_{r,\ell}}  \geq \beta_k/m_k \right) \\
&\leq  {\bf Prob} \left(\exists r \in [m_k] : \norm{  \sum_{\ell=1}^{\lfloor n/k \rfloor} \Y_{r,\ell}}  \geq \beta_k \left( \frac{ne}{k-1} \right)^{-(k-1)} \right) \\
 &\leq 2d \left( \frac{ne}{k-1} \right)^{k-1} \exp{\left(\frac{-\beta_k^2 (\frac{ne}{k-1})^{-2(k-1)} /2}{4\frac{n}{k}\left(\frac{L}{n}\right)^{2k}  + \beta_k (\frac{ne}{k-1})^{-(k-1)}\left(\frac{L}{n}\right)^{k}} \right)}\\
  &= 2d \left( \frac{ne}{k-1} \right)^{k-1} \exp{\left(\frac{-\beta_k^2 n (\frac{k-1}{e})^{2(k-1)} /2}{4L^{2k}/k  + \beta_k (\frac{k-1}{e})^{(k-1)} L^{k}} \right)}
\end{align*}
Set 
\begin{equation}
\label{beta}
\beta_k = \frac{4}{\sqrt{n}}(\frac{e}{k-1})^{k-1} L^k \sqrt{\log(2d) + (k-1) \log( ne / (k-1)) - \log(\delta)}
\end{equation}
Under the assumption that 
\begin{equation}
\label{assume:k}
k \leq \frac{4\sqrt{n}}{\sqrt{\log(d) + (k-1) \log( ne / (k-1)) - \log(\delta)}},
\end{equation}
which is implied by the stated condition \eqref{k:simplified} on $k$, 
it follows that  
$$\beta_k (\frac{k-1}{e})^{k-1} L^{k} \leq 4L^{2k}/k, $$ and so we can continue to bound  
\begin{align*}
{\bf Prob} \left(\sum_{r=1}^{m_k}\norm{  \sum_{\ell=1}^{\lfloor n/k \rfloor} \Y_{r,\ell}}  \geq \beta_k \right) &\leq 
2d \left( \frac{ne}{k-1} \right)^{k-1} \exp{\left(\frac{-\beta_k^2 n (\frac{k-1}{e})^{2(k-1)} /2}{8L^{2k}/k} \right)} \nonumber \\
&\leq 2d \left( \frac{ne}{k-1} \right)^{k-1} \exp{\left( -k \left(\log(d) + (k-1) \log( ne / (k-1)) - \log(\delta) \right) \right)} \nonumber \\
  &\leq \delta^k
\end{align*}
Thus, we conclude that for each $k=1,2,\dots, $ satisfying assumption \eqref{assume:k}, it holds that 

$${\bf Prob}[ \| \Z_{k} - \E{\Z_{k}}  \| > \beta_k + \| \mtx{D}_k \|] \leq \delta^k$$
Recalling $$\| \mtx{D}_k \| \leq  \frac{2L(k-1)}{n}\left(\frac{eL}{k-1} \right)^{k-1}$$ yields the result.
\end{proof}

\section{Proof of Theorem \ref{thm:main}}

We can bound the error $ \| \Z_n - \mathbb{E}[\Z_n] \| $ from Theorem \ref{thm:main} by combining Proposition \ref{prop:2} with Lemma \ref{prop:1}.
\begin{corollary}
\label{prop:3}
Suppose that $L$, $n,$ and $\delta \in (0, 1/2]$ are such that
\begin{equation}
\label{eq:restriction}
\max\{3, \lceil{ L e^2 \rceil} \} \leq  \log(n) +1 \leq \left( \frac{16 n}{\log(d/\delta) + \log(ne)} \right)^{1/3}
\end{equation}
Then with probability exceeding $1 -2\delta$, 
\begin{align*}
 \| \Z_n - \mathbb{E}[\Z_n] \| &\leq  \frac{2L  e^L \log(n)}{\sqrt{n}} \left( 2 \sqrt{ \log( 2d (ne)^{n} / \delta) }+  \frac{\log(n)}{\sqrt{n}} \right) \end{align*}
\end{corollary}

\begin{proof}
First, $\| \Z_n - \mathbb{E}[\Z_n] \| \leq \sum_{k=1}^{n} \norm{\Z_{n,k} - \E{\Z_{n,k}}}$ by the triangle inequality. 
By Proposition \ref{prop:1}, 
$$
\sum_{k=\lceil \log(n) \rceil}^{n} \norm{\Z_{n,k} - \E{\Z_{n,k}}} \leq \frac{2eL}{n} \quad \quad \text{with probability 1}.
$$  
Now, given \eqref{eq:restriction}, we can apply Proposition \ref{prop:2} to each of $k=1, 2, \dots, \lceil \log(n) \rceil$, and via the union bound, we obtain that the following holds with probability at least $1 -  \sum_{k=1}^{\lceil \log(n) \rceil} \delta^{-k} \geq 1 - \frac{\delta}{1-\delta} \geq 1 - 2\delta:$
 \begin{align*} \sum_{k=1}^{\lceil \log(n) \rceil} \norm{\Z_{n,k} - \mathbb{E}(\Z_{n,k})} &\leq \sum_{k=1}^{\lceil \log(n) \rceil} \gamma_k  \\
 &\leq 2 \sum_{k=1}^{\lceil \log(n) \rceil} \left( \frac{e L}{k-1} \right)^{k-1}  \left( \frac{2L}{\sqrt{n}} \sqrt{ \log( 2d (ne / (k-1))^{k-1} / \delta) }+  \frac{L(k-1)}{n} \right)
 \\
  &\leq 2L \left( \frac{2}{\sqrt{n}} \sqrt{ \log( 2d (ne)^{n} / \delta) }+  \frac{\log(n)}{n} \right) \sum_{k=1}^{\lceil \log(n) \rceil} \left( \frac{e L}{k-1} \right)^{k-1} \nonumber \\
  &\leq \frac{2L  e^L \log(n)}{\sqrt{n}} \left( 2 \sqrt{ \log( 2d (ne)^{n} / \delta) }+  \frac{\log(n)}{\sqrt{n}} \right) 
   \end{align*}
where in the final inequality, we use that $\left( \frac{e L}{x} \right)^{x} $ is maximized over $x > 0$ at $x^{*} = L$. 
We have the stated result. 

\end{proof}

\noindent {\emph{Proof of Theorem \ref{thm:main}} from Corollary \ref{prop:3}}. Write $\| \Z_n  - e^{\X} \| \leq \| \Z_n - \mathbb{E}[\Z_n] \| + \| \mathbb{E}[\Z_n] - e^{\X} \|$.  Bound $\| \Z_n - \mathbb{E}[\Z_n] \| $ using Corollary \ref{prop:3} and bound $\| \mathbb{E}[\Z_n] - e^{\X} \| =  \| (I + \frac{1}{n}\X)^n - e^{\X} \|$ using Lemma \ref{prop:4} to arrive at the statement of Theorem \ref{thm:main}.

\section{Conclusion and Future Directions}
We derived a large deviations bound for the convergence rate of a certain type of product of random matrices toward its limiting distribution. Our results are quite general and nearly sharp with respect to dependence on the matrix size $d$ and number of terms in the product, $n$. 

One particularly immediate application of our rates of convergence is in the analysis of random matrix products arising in stochastic iterative algorithms such as Oja's algorithm for streaming principal component analysis \cite{streaming2}.  
One area of future work would be to use our results to derive convergence rates for Oja's method using minimal assumptions -- an area of ongoing research (see, for example, \cite{2016Zhu, 2016Jain}).
This is particularly important because of the fundamental role streaming PCA plays in high-dimensional data analysis.

\bibliographystyle{alpha}
\bibliography{RandMatrix}

\appendix

{\section{Proof of \cref{prop:1}} \label{app:1}}

\tailbound*

\begin{proof}
We have that 
\begin{align*}
\norm{\Z_{n,k} - \E{\Z_{n,k}}} &=\frac{1}{n^k}  \norm{\sum_{1 \leq j_1 < \cdots < j_k \leq n} \left( \X_{j_k}\X_{j_{k-1}} \cdots \X_{j_1} - \X^k\right)} \\
&\leq \frac{1}{n^k}  \sum_{1 \leq j_1 < \cdots < j_k \leq n} \norm{ \X_{j_k}\X_{j_{k-1}} \cdots \X_{j_1} - \X^k }\\
&\leq \frac{1}{n^k}  \sum_{1 \leq j_1 < \cdots < j_k \leq n} \norm{ \X_{j_k}\X_{j_{k-1}} \cdots \X_{j_1}} +\norm{ \X^k}\\
 \text{(Using that $\norm{\X_j} \leq L$)} \quad \quad \quad \quad &\leq \frac{2L^k}{n^k}  \binom{n}{k}\\ 
&\leq  \frac{2L^k}{n^k} \left(\frac{en}{k}\right)^k\\
&= 2 \left(\frac{L e}{k}\right)^k
\end{align*}
Hence it remains to show that
$$\sum_{k= \lceil \log(n) \rceil\}}^n \left( \frac{L e}{k}\right)^k \leq \frac{L e^2}{n (e-1)} $$
Let $k_0=\lceil \log(n) \rceil$ in the remainder. First, we observe that  $\left( \frac{Le}{k_0}\right)^{k_0} \leq \frac{Le}{n}$:
\begin{align*}
\left( \frac{Le}{k_0}\right)^{k_0} \leq \frac{Le}{n} \quad &\Leftrightarrow \quad k_0\log(Le) - k_0\log(k_0) \leq \log(Le)-\log(n)\\
&\Leftrightarrow \quad (k_0 - 1)\log(Le) - k_0\log(k_0) \leq -\log(n)\\
&\Leftrightarrow \quad (k_0 - 1)\log(Le) - k_0(\log(k_0)-1) - k_0 \leq -\log(n)\\
\end{align*}
Since $k_0 \geq \log(n),$ it suffices to show that 
$$(k_0 - 1)\log(Le) - k_0(\log(k_0) - 1) \leq 0$$
We consider two cases:
\begin{enumerate}
    \item Case 1: If $L \leq \frac{1}{e}$, then $(k_0 - 1) \log(Le) \leq (k_0 - 1) \log(1) \leq 0.$  Thus we require $-k_0(\log(k_0) - 1) \leq 0$. 
 This clearly holds because $k_0 \geq e$. 
    \item Case 2: If  $L > \frac{1}{e}$, then $\log(Le) > \log(1) = 0 \Rightarrow -\log(Le) < 0$. Since $k_0 \geq Le^2,$ it follows that
    \begin{align*}
    (k_0 - 1)\log(Le) - k_0(\log(k_0) - 1) &\leq (k_0 - 1)\log(Le) - k_0(\log(Le^2) - 1) \\
    &= (k_0 - 1)\log(Le) - k_0(\log(Le) + 1 - 1) \\
    &= -\log(Le) \\
    &< 0.\\
    \end{align*}
\end{enumerate}
Now, for each $k \geq k_0$, 
\begin{comment}
\begin{align*}
    \left(\frac{Le}{k+1}\right)^{k+1} \leq \left(\frac{Le}{k}\right)^{k+1}  &\leq \frac{Le}{k}\left(\frac{Le}{k}\right)^{k}\\
    \leq \frac{Le}{Le^2}\left(\frac{Le}{k}\right)^k &= e^{-1}\left(\frac{Le}{k}\right)^k
    \end{align*}
\end{comment}
$$\left(\frac{Le}{k+1}\right)^{k+1} \leq \left(\frac{Le}{k}\right)^{k+1}  \leq \frac{Le}{k}\left(\frac{Le}{k}\right)^{k}
    \leq \frac{Le}{Le^2}\left(\frac{Le}{k}\right)^k = e^{-1}\left(\frac{Le}{k}\right)^k $$
By induction, it follows that 
\begin{align*}
    \left(\frac{Le}{k+1}\right)^{k+1} \leq e^{k_0 - k}\left(\frac{Le}{k_0}\right)^{k_0} &\leq e^{k_0 - k}\frac{Le}{n};
\end{align*}
Hence,
\begin{align*}
    \sum_{k=k_0}^n \left(\frac{Le}{k}\right)^{k}  \leq \frac{Le}{n}\sum_{k=k_0}^n e^{k_0-k} &\leq \frac{Le}{n}\sum_{k=0}^\infty e^{-k} = \frac{Le^2}{n(e-1)}
\end{align*}
\end{proof}

\section{Proof of Lemma \ref{prop:4}}
\label{app:3}

\eXbound*

\begin{proof}
The proof uses only basic analytic tools and inequalities.  Recall the matrix exponential: $e^{\X} :=  \sum_{k=0}^\infty \frac{\X^k}{k!}$.  Let $\sigma = \| \X \|$. Then we have
\begin{align*}
\norm{(\I+\frac{1}{n}\X)^n - e^{\X}} &= \norm{\sum_{k=0}^n \binom{n}{k} \frac{\X^k}{n^k} - \sum_{k=0}^\infty \frac{\X^k}{k!}}\\
&= \norm{\sum_{k=0}^n\frac{\X^k}{k!}\left(\frac{n!}{n^k(n-k)!} - 1 \right) - \sum_{k=n+1}^\infty \frac{\X^k}{k!}}\\
&\leq \norm{\sum_{k=0}^n\frac{\X^k}{k!}\left(\frac{n!}{n^k(n-k)!} - 1 \right)} + \norm{\sum_{k=n+1}^\infty \frac{\X^k}{k!}}\\
&\leq \sum_{k=0}^n\frac{\norm{\X}^k}{k!}\left|\left(\frac{n!}{n^k(n-k)!} - 1 \right)\right| + \sum_{k=n+1}^\infty \frac{\norm{\X}^k}{k!}\\
&= \sum_{k=0}^n\frac{\sigma^k}{k!}\left(1-\frac{n!}{n^k(n-k)!} \right) + \sum_{k=n+1}^\infty \frac{\sigma^k}{k!}\\
&= e^{\sigma} - \sum_{k=0}^n \binom{n}{k}\frac{\sigma^k}{n^k}\\
&= e^{\sigma} - \left( 1 + \frac{\sigma}{n} \right)^n\\
&=  e^{\sigma}- \exp{({n \log(1 + \sigma/n)})} \\
&\leq e^{\sigma} -\exp\left( \sigma - \frac{1}{2}\frac{\sigma^2}{n} \right) \\
&=  e^{\sigma} \left( 1 -  \exp{( -\frac{\sigma^2}{2n})} \right) 
\end{align*}
where in the final inequality, we used that $\log(1 + \sigma/n) \geq \frac{\sigma}{n} - \frac{1}{2}\left(\frac{\sigma}{n} \right)^2$.  Thus, using also that $e^{-x} \geq 1 - x$ for all $x > 0,$
\begin{align*}
\norm{(\I+\frac{1}{n}\X)^n - e^{\X}}  &\leq    e^{\sigma} \frac{\sigma^2}{2n}
\end{align*}
\end{proof}

%Then $n = kL + \frac{L(L-1)}{2}$

\end{document}